\documentclass[a4paper, 12pt]{amsart}
\usepackage{amsmath}
\usepackage{amssymb, latexsym, amsthm}

\newtheorem{Thm}{Theorem}   [section]
\newtheorem{Prop}[Thm]{Proposition}
\newtheorem{Cor}[Thm]{Corollary}
\newtheorem{Lem}[Thm]{Lemma}

\theoremstyle{definition}
\newtheorem*{definition}{Definition}
\newtheorem{remark}{Remark}[section]

\numberwithin{equation}{section}

\begin{document}

\newcommand{\Z}[0]{\mathbb{Z}}
\newcommand{\Q}[0]{\mathbb{Q}}
\newcommand{\F}[0]{\mathbb{F}}
\newcommand{\N}[0]{\mathbb{N}}
\newcommand{\p}[0]{\mathfrak{p}}
\newcommand{\m}[0]{\mathrm{m}}
\newcommand{\n}[0]{\mathrm{n}}
\newcommand{\Tr}{\mathrm{Tr}}
\newcommand{\Hom}[0]{\mathrm{Hom}}
\newcommand{\Gal}[0]{\mathrm{Gal}}
\newcommand{\Res}[0]{\mathrm{Res}}
\newcommand{\id}{\mathrm{id}}
\newcommand{\mult}{\mathrm{mult}}
\newcommand{\adm}{\mathrm{adm}}
\newcommand{\tr}{\mathrm{tr}}
\newcommand{\pr}{\mathrm{pr}}
\newcommand{\Ker}{\mathrm{Ker}}
\newcommand{\ab}{\mathrm{ab}}
\newcommand{\sep}{\mathrm{sep}}
\newcommand{\triv}{\mathrm{triv}}
\newcommand{\alg}{\mathrm{alg}}
\newcommand{\ur}{\mathrm{ur}}
\newcommand{\Coker}{\mathrm{Coker}}
\newcommand{\Aut}{\mathrm{Aut}}
\newcommand{\To}{\longrightarrow}
\renewcommand{\c  }{\mathcal }
\newcommand{\wt}{\widetilde}
\newcommand{\op}{\mathrm}
\newcommand{\ad}[0]{\mathrm{ad}}

\title{Ramification filtration and differential forms}
\author {Victor Abrashkin}
\address{Department of Mathematical Sciences, MCS, 
Durham University,  
South Rd, Durham DH1 3LE, United Kingdom \ \&\ Steklov 
Institute, Gubkina str. 8, 119991, Moscow, Russia
}
\email{victor.abrashkin@durham.ac.uk}
\date{\today }
\keywords{local field, Galois group, ramification filtration}
\subjclass[2010]{11S15, 11S20}

\begin{abstract} Let $L$ be a complete discrete 
valuation field of prime characteristic $p$ with finite residue field. 
Denote by $\Gamma _{L}^{(v)}$ the ramification subgroups of 
$\Gamma _{L}=\operatorname{Gal}(L^{sep}/L)$. We consider the category 
$\operatorname{M\Gamma }_{L}^{Lie}$ of finite 
$\mathbb{Z}_p[\Gamma _{L}]$-modules $H$, satisfying some additional 
(Lie)-condition on the image of $\Gamma _L$ in 
$\operatorname{Aut}_{\mathbb{Z}_p}H$. In the paper it is proved that 
all information about the images of the groups $\Gamma _L^{(v)}$ in 
$\operatorname{Aut}_{\mathbb{Z}_p}H$ 
can be explicitly extracted from some differential forms $\widetilde{\Omega} [N]$ on 
the Fontaine etale $\phi $-module $M(H)$ associated with $H$. The 
forms $\widetilde{\Omega} [N]$ are completely determined by a canonical connection $\nabla $ 
on $M(H)$. In the case of fields $L$ of mixed characteristic, which contain a primitive 
$p$-th root of unity, we show that a  
similar problem for $\mathbb{F}_p[\Gamma _L]$-modules also admits a solution.  
In this case we use the field-of-norms functor to 
construct the corresponding $\phi $-module together with the action 
of the Galois group of a cyclic extension $L_1$ of $L$ of degree $p$. 
Then our solution involves the characteristic $p$ part (provided by 
the field-of-norms functor) and the condition for a ``good'' lift of 
a generator of $\operatorname{Gal}(L_1/L)$. Apart from the 
above differential forms the statement of this condition uses the power series 
coming from the $p$-adic period of the formal group $\mathbb{G}_m$. 
\end{abstract}

\maketitle

\section*{Introduction}

Let $L$ be a complete discrete valuation field with finite residue field of 
characteristic $p$. Let $\Gamma _L$ be the absolute Galois group of $L$. 
Let $\{\Gamma _L^{(v)}\}_{v\geqslant 0}$ be the  
filtration of $\Gamma _L$ by the 
ramification subgroups,  \cite{Se1}. This filtration provides $\Gamma _L$ with additional structure and allows us to   
introduce various classes of infinite field extensions 
(arithmetically profinite, deeply ramified etc.), which  play an important 
role in 
modern arithmetic algebraic geometry. For $\Gamma _L$-modules $H$, the evaluation of 
$v_0(H)\in\Q $ such that $\Gamma _L^{(v)}$ act trivially on $H$ for $v>v_0(H)$, provides us with good estimates for discriminants of the fields of definition 
of $h\in H$. Such estimates are used very often 
to answer various number theoretic questions. 

However, an explicit description of the structure of 
the ramification filtration 
for a very long time was known only at 
the level of the Galois groups of abelian field extensions. 
At the time when the structure of the Galois group 
$\Gamma _L$ was completely described (the case of the maximal $p$-extensions --
Shafarevich\,\cite{Sh}, Demushkin\,\cite{Dm}, and the 
general case -- Janssen-Wingberg\,\cite{JW}) it became clear that  
$\Gamma _L$ is a very weak 
invariant of the field $L$. The situation cardinally changed later 
when it was established  
(Mochizuki\,\cite{Mo}, author\,\cite{Ab6}) that taking 
$\{\Gamma _L^{(v)}\}_{v\geqslant 0}$ into account 
gives us absolute invariant of the field $L$. 
However, in order to work with this invariant we still 
need to know an explicit description of the filtration by 
the groups $\Gamma _L^{(v)}$. 

I.R.\,Shafarevich always paid attention to this problem, for example, 
cf.\,Introduction to \cite{Ko}. 
His motivation was the following: for every prime number $p$ there is only one 
such filtration and we know almost nothing about its structure. In 1990's  
the author developed a nilpotent version of the Artin-Schreier theory and 
obtained an explicit description of the ramification filtration modulo the 
subgroup of $p$-th commutators of $\Gamma _L$. Such description was obtained,  
first, in the characteristic $p$ case, \cite{Ab1, Ab2, Ab3}, 
and then developed in 
the mixed characteristic case, \cite{Ab12, Ab13, Ab14}. These results play a 
crucial role in this paper, where we study the images of the ramification subgroups  
$\Gamma _L^{(v)}$ in the group of automorphisms 
$\op{Aut}_{\Z _p}H$ of finite $\Z _p[\Gamma _L]$-modules $H$. 
Our main result states that  
this arithmetic structure can 
be completely described (under some additional condition) in  
purely geometric properties of Fontaine's etale 
$\phi $-modules $M(H)$.

More precisely, if $\op{char}\,L=p$ 
we construct differential forms $\wt{\Omega} [N]$, 
$N\in\N $, on an extension of scalars of 
$M(H)$, and specify the way how the image of the ramification filtration 
in $\op{Aut}_{\Z _p}H$ can be recovered from these forms. 
Note that the definition of $\wt{\Omega} [N]$ depends only on a 
natural connection constructed on $M(H)$. 
If $\op{char}L=0$ we assume that $L$ contains a $p$-th root 
of unity $\zeta _1\ne 1$ and restrict ourselves to the case of 
Galois $\F _p$-modules. Then we apply the field-of-norms functor 
to reduce the situation to the characteristic $p$ case and use a 
characterization of ``good'' lifts of automorphisms of our cyclic field 
extension of $L$ from \cite{Ab12, Ab13}. This characterization uses 
again the differential forms $\wt{\Omega} [N]$ and a power series coming 
from the $p$-adic period 
of the formal group $\mathbb{G}_m$. 

By our opinion this result establishes quite interesting link between the   
Galois theory of local fields and very popular area of $D$-modules, 
lifts of Frobenius, Higgs vector bundles etc.

\section{Statement of the main result} \label{S1} 

\subsection{General notation} \label{S1.1} 
Everywhere in the paper $p$ is a fixed prime number. 
If $E_0$ is a field then $E_0^{sep}$ is its   
separable closure in some  algebraic closure 
$E_0^{alg}$ of $E_0$. 
If $E$ is a field such that 
$E_0\subset E\subset E_0^{sep}$ set  
$\Gamma _{E}=\op{Gal}(E_0^{sep}/E)$. The field $E_0^{sep}$ will be 
considered as a left $\Gamma _{E_0}$-module, i.e. 
for any $\tau _1,\tau _2\in\Gamma _{E_0}$ and $o\in E_0^{sep}$, 
$(\tau _1\tau _2)o=\tau _1(\tau _2o)$. If $\op{char}E_0=p$ we set 
for any $a\in E_0^{sep}$, 
$\sigma (a)=a^p$. 

If $V$ is a module over a ring $R$ then   
$\op{End}_RV$ is the $R$-algebra of $R$-linear endomorphisms of $V$.  
We always consider $V$ as a left $\op{End}_RV$-module, i.e.  
if $l_1,l_2\in\op{End}_RV$ and $v\in V$ then 
$(l_1l_2)v=l_1(l_2(v))$. We also consider $\op{End}_RV$ 
as a Lie $R$-algebra with the Lie bracket  
$[l_1,l_2]=l_1l_2-l_2l_1$. 
If $S$ is an $R$-module then we often denote by 
$V_S$ the extension of scalars 
$V\otimes _RS$.

\subsection{Functorial system of lifts to characteristic 0} \label{S1.2} 

Suppose $K_0=k_0((t_0))$ is the field of formal Laurent series 
in a (fixed) variable $t_0$ with coefficients in a finite field $k_0$ 
of characteristic $p$. 
The uniformiser $t_0$ provides a $p$-basis for any field 
extension $   E$ of $   K_0$ in $   K_0^{sep}$, i.e. the set 
$\{1,t_0,\dots ,t_0^{p-1}\}$ is a 
$E^p$-basis of $   E$. We use this $p$-basis  
to construct a compatible system of 
lifts $O(   E)$ of the fields $   E$ to characteristic 0. 
This is a special case of the construction of lifts  
from   \cite{Bth}; it can be explained as follows. 

For all $M\in\N $, set  $O_{M}(   E)=
W_{M}(\sigma ^{M-1}   E)[\bar t_0]$, where $\bar t_0=[t_0]$ is 
the Teichmuller representative of $t_0$ in the 
ring of Witt vectors $W_{M}(   E)$. 
The rings $O_M(   E)$ are the lifts of $   E$ modulo $p^M$, i.e. they are 
flat $\Z /p^M$-algebras such that $O_M(   E)\otimes _{\Z /p^M}\Z /p=   E$.  
Note that the system
$$\{O_M(   E)\ |\ M\in\N,    K_0\subset    E\subset    K_0^{sep}\}$$ 
is functorial in $M$ and $   E$. 
In particular, if $E/K_0$ is Galois then there is a natural action 
of $\op{Gal}(   E/   K_0)$ on $O_M(   E)$ 
and $O_M(   E)^{\op{Gal}(E/K_0)}=O_M(K_0)$. 
The morphisms $W_M(\sigma )$ induce $\sigma $-linear 
morphisms on $O_M(E)$ which will be denoted again by $\sigma $. 
In particular, $\sigma (\bar t_0)=\bar t_0^{\,\,p}$. 

Introduce the lifts of the above fields $   E$ 
to characteristic 0 by setting 
$O(   E)=\underset{M}{\varprojlim}\,O_M(   E)$. 
Then $O_M(   E)=O(   E)/p^M$ and 
$O(   E)[1/p]$ is a complete discrete valuation field 
with uniformiser $p$ and the residue field $   E$. 
Clearly, we have the induced morphism $\sigma $ on each $O(   E)$. 
Also, if $   E/   K_0$ is Galois then $\op{Gal}(   E/   K_0)$ 
acts on $O(   E)$ and 
$O(   E)^{\op{Gal}(   E/   K_0)}=O(   K_0)$. 
Notice that  
$O(   K_0)=\underset{M}{\varprojlim}\, W_M(k_0)((\bar t_0))$  
is the completion of 
the ring of formal Laurent series $W(k)((\bar t_0))$. 

Set $O_{sep}=O(   K_0^{sep})$.

The system of lifts 
$O(   E)$, $   E\subset K_0^{sep}$, can be extended to the system 
of lifts of all extensions of $   K_0$ in $K_0^{alg}$. Indeed, note that  
$   K_0^{alg}=\bigcup _{n\geqslant 0}K_0(t_n)^{sep}$, 
where $t_n^{p^n}=t_0$. Then $t_n$ gives the $p$-basis 
$\{1,t_n,\dots ,t_n^{p-1}\}$ for all 
separable extensions $E$ of $K_0(t_n)$ in 
$K_0^{alg}$ and we 
obtain (as earlier) 
the corresponding lifts $O(   E)$. The system of lifts $O(   E)$ 
is functorial in  
$   E\subset    K_0^{alg}$ 
(use that any separable extension $   E_n$ of $   K_0(t_n)$ 
appears uniquely as the composite $   E_0   K_0(t_n)$, 
where $   E_0/   K_0$ is separable). 
In particular, 
consider $   K_0^{rad}=\bigcup _{N\in\N }   K_0^{ur}(t_0^{1/N})$. 
Then for  
the above defined lift of $   K_0^{rad}$ we have 
$O(   K_0^{rad})=\bigcup _{N\in\N }O(   K_0^{ur})[\bar{t_0}^{1/N}]$, 
where $   K_0^{ur}=\bar k_0((t_0))$ is the maximal unramified extension of $   K_0$.

\subsection{Equivalence of the categories of $p$-groups and 
Lie algebras,\,  \cite{La}} \label{S1.3} 

Let $L$ be a finitely generated Lie $\Z _p$-algebra 
of nilpotent class $<p$, 
i.e.\,the 
ideal of $p$-th commutators $C_p(L)$ of $L$ is equal to zero. 
Let $A$ be an enveloping algebra  of $L$. Then the elements of 
$L\subset A$ generate the augmentation ideal $J$ of $A$.  
There is a morphism of $\Z _p$-algebras $\Delta :A\To A\otimes A$ 
uniquely determined by the 
condition $\Delta (l)=l\otimes 1+1\otimes l$ for all $l\in L$. 
Then the set ${\exp}(L)\,\op{mod}J^p$ is identified with the group of all 
''diagonal elements modulo degree $p$`` consisting of  
$a\in 1+J\,\op{mod}\,J^p$ such that 
$\Delta (a)
\equiv a\otimes a\,\op{mod}(J\otimes 1+1\otimes J)^p$.   
\medskip  

In particular,  there is a natural embedding 
$L\subset A/J^p$ and  
the identity  
$${\exp}(l_1)\cdot 
{\exp}(l_2)\equiv {\exp}(l_1  \circ l_2)\,\op{mod}\,J^p$$ 
induces 
the Campbell-Hausdorff composition law    
$$(l_1,l_2)\mapsto l_1  \circ l_2=
l_1+l_2+\frac{1}{2}[l_1,l_2]+\dots ,\ \ l_1,l_2\in L\, .$$  
This composition law provides the set $L$ with 
a group structure. We denote this group by $G(L)$. 
Clearly $G(L)\simeq {\exp}(L)\,\op{mod}\,J^p$.

With the above notation the functor $L\mapsto G(L)$ 
determines the  
equivalence of the categories of finitely generated 
$\Z _p$-Lie algebras and profinite $p$-groups of nilpotence class $<p$. 
Note that a subset 
$I\subset L$ is an ideal in $L$ iff 
$G(I)$ is a normal subgroup in $G(L)$.

\subsection{(Lie)-condition} \label{S1.4} 
For any finite field extension $   K$ of $   K_0$ in $   K_0^{sep}$, 
let $\op{M\Gamma }_{   K}$ be the category of finitely generated 
$\Z _p$-modules $   H$ with continuous left action of $\Gamma _{   K}$. 
Each element  $h\in H$ is defined over some finite extension $K(h)$ of $K$. 
In some sense the family of these fields determines ''arithmetic`` properties of $H$. 
More detailed information about the fields $K(h)$ 
can be obtained from the knowledge of the images of  
the ramification subgroups in upper numbering $\Gamma _{   K}^{(v)}$, $v>0$, in 
$\op{Aut}_{\Z _p}   H$. 
For example, the minimal number $v_0(H)\in\Q $ such that 
all $\Gamma _{K}^{(v)}$ with $v>v_0(   H)$  
act trivially on $H$ provides us with    
upper estimates for the discriminants of the fields of definition 
of $h\in H$, cf.\,\cite{Ab6}.  

Let $   H_0\in\op{M\Gamma }_{   K_0}$ and 
let $\pi _{H_0}:\Gamma _{K_0}\To \op{Aut}_{\Z _p}   H_0$ be the group 
homomorphism which determines the $\Gamma _{   K_0}$-module structure on $   H_0$. 
Consider the full subcategory $\op{M\Gamma }_{   K_0}^{\op{Lie}}$ in  
$\op{M\Gamma }_{   K_0}$ which consists of modules $   H_0$ satisfying the 
following condition: 
\medskip

{\bf (Lie)}\ {\it The image $I(   H_0):=\pi _{   H_0}(\c I)\subset 
\Aut _{\Z _p}(H_0)$  of the wild inertia 
subgroup $\c I\subset\Gamma _{K_0}$ appears in the form 
${\exp}(L(H_0))$, where 
$L(H_0)\subset\op{End}_{\Z _p}H_0$ is a Lie subalgebra such that 
$L(H_0)^p=0$.}
\medskip

The condition $L(   H_0)^p=\{l_1\ldots l_p\ |\ l_1,\ldots ,l_p\in L(   H_0)\}=0$ 
(the product is taken in $\op{End}_{\Z _p}   H_0\supset L(   H_0)$) 
implies that $L(   H_0)$ is a finitely generated 
nilpotent $\Z _p$-algebra Lie of nilpotence class $<p$. 
This gives the group isomorphism $\exp :G(L(   H_0))
\simeq I(H_0)$. 
Note that 
any normal subgroup of $I(H_0)$ appears in the form $\exp G(J)$, 
where $J$ is a Lie ideal of $L(H_0)$.

\begin{remark}
 If $p   H_0=0$ and $\op{dim}_{\F _p}   H_0\leqslant p$ 
 then $   H_0$ is of the Lie type. 
 \end{remark}

\subsection{The first main result: the characteristic $p$ case} \label{S1.5}

Suppose $   H_0\in\op{M\Gamma }_{K_0}^{\op{Lie}}$. 
Our target is to determine for all $v> 0$,  
the images 
$\pi _{   H_0}(\Gamma _{   K_0}^{(v)})$ of the ramification subgroups 
$\Gamma _{K_0}^{(v)}$ via an explicit construction 
of the ideals $L(H_0)^{(v)}\subset L(H_0)$ such that 
$\exp (L(H_0)^{(v)})=
\pi _{H_0}(\Gamma _{K_0}^{(v)})$. 

Our  approach uses  
Fontaine's ``analytical'' description of the 
Galois modules  
$H_0\in\op{M\Gamma }_{   K_0}^{\op{Lie}}$ in 
terms of etale $(\phi ,O(K_0))$\,-modules $M(H_0)$. 
A geometric nature of $M(H_0)$ is supported 
by the existence of an analogue of the classical 
connection  
$\nabla :M(H_0)\To  M(H_0)\otimes _{O(K_0)}
\Omega ^1_{O(K_0)}$, 
cf.\,  \cite{Fo}. 
(This map is uniquely characterized by the condition 
$\nabla\cdot\phi =(\phi \otimes\phi )\cdot \nabla $\,.) 
The required information about the behaviour of ramification subgroups  
can be then extracted from some differential forms 
$$ \wt{\Omega }[N]\in M(H_0)_{O(K_0^{rad})}
\otimes_{O(K_0)}
\Omega ^1_{O(K_0)}$$ 
The construction of these differential forms 
is given completely in terms of the above connection $\nabla $ 
and can be explained as follows.

Let $   K\subset    K_0^{sep}$ be a fixed tamely ramified 
finite extension of $K_0$ such that 
$\pi _{H_0}(\Gamma _{K})=I(H_0)$. 
Then $H:=H_0|_{\Gamma _{K}}$ can be described via an etale 
$(\phi ,O(K))$-module 
$M(H)=M(H_0)\otimes _{O(K_0)}O(K)$. 
Recall that $M(H)=
(H\otimes _{\Z _p}O_{sep})^{\Gamma _{K}}$ is a finitely generated 
$O(K)$-module 
with a $\sigma $-linear morphism 
$\phi :M(H)\To M(H)$ such that the image 
$\phi (M(H))$ generates $M(H)$ over $O(K)$. 
This allows us to identify 
the elements of $H$ with a set of $O_{sep}$-solutions 
of a suitable  system of equations with coefficients in $O(K)$. 

We establish below the construction of $M(H)$ by introducing a  
$\Z _p$-linear embedding $\c F:   H\To M( H)$ which induces 
by extension of scalars the identification 
$H_{O(K)}\simeq M(H)$. (We denote this identification by  
the same symbol $\c F$.) 

Now let $\wt{B}$ be (a inique) $O(K)$-linear operator on 
$M(H)$ such that 
for any $m\in \c F(H)$, $\nabla (m)=\wt{B}(m)d{\bar t}/\bar t$. 
Then for every $N\in\Z _{\geqslant 0}$ we introduce 
the differential forms 
$$\wt{\Omega }[N]=\phi ^N\wt{B}\phi ^{-N}d\bar t/\bar t\in
\op{End}M(H)_{O(K^{rad})}\otimes _{O(K)}\Omega ^1_{O(K)}\,$$

Now we can use the identification $\c F:H_{O(K)}\simeq 
M(H)$ to 
obtain the corresponding differential forms $\Omega [N]$ on 
$\op{End}(H)_{O(K^{rad})}$ and to verify that 
$$\Omega [N]\in L(H)_{O(K^{rad})}
\otimes _{O(K)}\Omega ^1_{O(K)}=
L(H_0)_{O(K_0^{rad})}
\otimes _{O(K_0)}\Omega ^1_{O(K_0)}\, .$$
\begin{remark}
 Our differential forms will usually appear in the form 
 $\Omega =F\cdot d\bar t_0/\bar t_0$, where $F\in L(H)_{K_0^{rad}}$. 
 Then we set by definition 
 $$(\id _{L(H)}\otimes \sigma )\Omega =
 (\id _{L(H)}\otimes\sigma )F\cdot d\bar t_0/\bar t_0\,.$$
\end{remark}

Our first main result 
can be stated as follows. 

\begin{Thm} \label{T1.1} Suppose 
 $H_0\in\op{M\Gamma }_{K_0}^{\op{Lie}}$ is finite. 
 Then there is $N_0(H_0)\in\N $ 
 such that 
 for any (fixed) $N\geqslant N_0(H_0)$ 
 the following propery holds: 
 \medskip 
 
 if $(\id _{L(H)}\otimes\sigma ^{-N})\Omega [N]=
 \sum _{r\in\Q }\bar t_0^{\,-r}l_{r}d\bar t_0/\bar t_0$,  
 where all $l_r\in L(   H_0)_{W(\bar{k}_0)}$,  
 then the ideal $L(   H_0)^{(v)}$ is the minimal ideal in 
 $L(   H_0)$ such that for all $r\geqslant v$, 
 $l_r\in L(   H_0)^{(v)}_{W(\bar{k}_0)}$. 
\end{Thm}

 \begin{Cor} \label{C1.2} 
  If $v_0(H_0)=\max\{r\ |\ l_r\ne 0\}$ 
 then the ramification subgroups 
 $\Gamma _{K_0}^{(v)}$ act trivially on $H_0$ iff $v>v_0(H_0)$. 
 \end{Cor}
 
 \begin{remark}
  The construction of $\Omega [N]$ almost does not depend on the 
  choice of the tamely ramified 
  finite field extension $   K$ of $   K_0$. It depends essentially on 
  the choice of the uniformising element $t_0$ in $   K_0$ 
  and a compatible system of 
  $\alpha (k)\in W(k)$, where $[k:k_0]<\infty $, such that the trace of 
  $\alpha (k)$ in the field extension $W(k)[1/p]/{   K}_0$ equals 1. 
 \end{remark}
 
 \begin{remark}
  If $H_0$ is not $p$-torsion Theorem \ref{T1.1} can be applied to the 
  factors $H_0/p^M$ and our result describes the structure of 
  the images of $L(H_0)^{(v)}$ in all $L(H_0)/p^M$. 
 \end{remark}
 
 \subsection{The second main result: the mixed characteristic case} \label{S1.6}

Let $E_0$ be a finite field extension of $\Q _p$ with residue field $k_0$ 
and a uniformizing element $\pi _0$. Assume that $E_0$  
contains a 
$p$-th primitive root of unity $\zeta _1$. Consider the category 
$\op{M\Gamma }_{E_0,1}^\op{Lie}$ of finitely generated  
$\F _p[\Gamma _{E_0}]$-modules which satisfy a direct analog of 
the Lie condition from Sec.\,\ref{S1.4}. 

Take the infinite arithmetically profinite 
field extension $\wt{E}_0$ obtained from $E_0$ by joining all 
$p$-power roots of $\pi _0$. Then the theory of 
the field-of-norms functor $X$ provides us with the 
complete discrete valuation field of characteristic $p$, 
$X(\wt{E}_0)=K_0$, which has the same residue field and 
the uniformizing element $t_0$ obtained from the 
$p$-power roots of $\pi _0$. 
The functor $X$ also provides us with the identification of Galois groups 
$\Gamma _{K_0}=\Gamma _{\wt{E}_0}\subset \Gamma _{E_0}$. 

If $H_{E_0}\in\op{M\Gamma }^{\op{Lie}}_{E_0,1}$ then we obtain 
$H_0:=H_{E_0}|_{\Gamma _{K_0}}
\in\op{M\Gamma }_{K_0}^{\op{Lie}}$. 
As earlier, 
take a finite tamely ramified extension $K$ of $K_0$ 
(it corresponds to a unique tamely ramified extension $E$ of $E_0$ 
with uniformizer $\pi $ such that $\pi ^{e_0}=\pi _0$, 
where $e_0$ is the ramification index of $E/E_0$,   
and construct $(\phi ,O(K))$-module $M(H)$. This module 
inherits the action of 
$\op{Gal}(E(\root p\of {\pi })/E)=
\langle\tau _0\rangle ^{\Z /p}$. (Here $\tau _0$ is 
such that $\tau _0(\root p\of {\pi })=
\zeta _1\root p\of {\pi }$.) This situation
was considered in all details in the papers \cite{Ab12, Ab13}. In particular, in those papers 
we gave a characterization of the ``good '' 
lifts $\hat\tau _0$ of $\tau _0$. By definition, 
$\hat\tau _0\in\Gamma _{E}$ is ``good'' if its restriction to $H_{E_0}$ 
belongs to the image of the ramification 
subgroup $\Gamma _{E}^{(e^*)}$. 
Here $e^*:=pe/(p-1)$ and $e=e(E/\Q _p)$ 
is the ramification index of $E/\Q _p$. (This makes sense because 
$\tau _0\in \op{Gal}(E(\root p\of {\pi }/E)^{(e^*)}$.) Note that the field-of-norms functor is 
compatible with ramification filtrations on $\Gamma _{E_0}$ and $\Gamma _{K_0}$. Therefore, 
the knowledge of "good lifts" $\hat\tau _0$ together with Theorem \ref{T1.1} gives a complete description of 
the image of the ramification filtration of $\Gamma _{E_0}$ in $\op{Aut}H_{E_0}$. 

In \cite{Ab13} we proved that the action of $\hat\tau _0$ 
appears from an action of a formal group scheme of order $p$.  
As a result, the lift $\hat\tau _0$ is completely determined by the value 
$d\hat\tau _0(0)\in L(H)$ of its differential at 0, and we can use the characterization of differentials of ``good '' 
lifts from \cite{Ab13}, Theorem 5.1. 

Namely, let us first 
specify our $p$-th root of unity  
$$\zeta _1=1+\sum _{j\geqslant 0}[\beta _j]
\pi ^{(e^*/p)+j}\,\op{mod}\,p\,$$ 
(here all $[\beta _j]$ are the Teichmuller 
representatives of elements from the residue field of $E$).
Then we introduce the power series 
$\omega (t)\in O(K)$ such that 
$$1+\sum _{j\geqslant 0}\beta _j^pt^{e^*+pj}=
\wt{\exp}(\omega (t)^p)\,,$$
(here $\wt{\exp}$ is the truncated exponential). 
The series $\omega (t)^p$ is a kind of approximation of the 
$p$-adic period of the formal multiplicative group, which appears usually in 
explicit formulas for the Hilbert symbol, e.g. \cite{AJ}. In other words, we obtain another 
geometric condition characterizing "arithmetic" of the $\Gamma _{E_0}$-module $H_{E_0}$. 

\begin{Thm} \label{T1.3} 
 The lift $\hat\tau _0$ is ``good'' iff  
 $$d\hat\tau _0(0)\equiv \sum _{m\geqslant 0}
 \op{Res}\,\left (\omega (t)^{p^{m+1}}\Omega [m]\right )
 \,\op{mod}\,L(H)_{k}^{(e^*)}\,.$$
\end{Thm}

\begin{remark}
 Notice that the power series $\omega (t)^p$ has non-zero coefficients 
 only for the powers $t^{e^*+pj}$ and all these exponents $\geqslant e^*$. 
 Therefore, 
 the differential forms $\Omega [m]$ contribute to 
 the right hand side only via the images of 
 $\c F^0_{e^*+pj,-m}t^{-(e^*+pj)}$ in $L(H)_k$. But for $m\gg 0$, these images 
 belong to the images of the ramification ideals  
 $\c L^{(e^*)}_k$ and, therefore, disappear modulo 
 $L(H)_{k}^{(e^*)}$, and the sum in the right hand side is, as a matter of fact, finite.
\end{remark}

\section{$\phi $-module $M(H)$}  \label{S2}

\subsection{Specification of $\log \pi _{H}:
\Gamma _{K}\To G(L(H))$} \label{S2.1} 

As earlier, $H_0\in\op{M\Gamma }_{K_0}^{\op{Lie}}$,  
$K$ is a finite tamely ramified extension of 
$K_0$ in $K_0^{sep}$ such that $\pi_{   H_0}
(\Gamma _{   K})=I(H_0)$, 
$   H=H_0|_{\Gamma _{K}}$. 
Set $L(H)=L(H_0)$, $\pi _{H}=
\pi _{H_0}|_{\Gamma _{K}}$.

Consider the continuous group epimorphism 
$l_{H}:=\log (\pi _{H}):\Gamma _{K}\To G(L(   H))$. 
Since the $p$-group $G(L(   H))$ has nilpotence class $<p$ this epimorphism 
can be described in terms of the covariant version 
of the nilpotent Artin-Schreier theory 
from \cite{Ab2}. 
Namely, 
there are $e\in L(   H)_{O(   K)}$  and 
$f\in L(   H)_{O_{sep}}$ such that 
$(\id _{L(H)}\otimes\sigma )(f)=
e  \circ f$ and for any 
$\tau\in\Gamma _{   K}$, $l_{   H}(\tau )=
(-f)  \circ (\id_{L(H)}\otimes\tau )f$. 

It could be easily verified that $l_{H}$ is a group homomorphism. 
Indeed, 
$$l_{H}(\tau _1\tau _2)=
(-f)  \circ (\id_{L(H)}\otimes \tau _1\tau _2)f=$$
$$(-f)  \circ (\id_{L(H)}\otimes\tau _1)f  \circ (-f)  \circ 
(\id_{L(H)}\otimes\tau _2)f=
l_{H}(\tau _1)\circ l_{H}(\tau _2)\, ,$$
because $(\id_{L(H)}\otimes\tau _1)l_{H}(\tau _2)=l_{H}(\tau _2)$ 
\medskip 

{\bf Notation.} We will use below the following notation: $\sigma _H=\id _H\otimes\sigma $ and 
$\sigma _{L(H)}=\id _{L(H)}\otimes \sigma $. For example, if 
$u=\sum _{\alpha }h_{\alpha }\otimes o_{\alpha }$, where 
all $h_{\alpha }\in H$ and $o_{\alpha }\in O(K)$ then 
$\sigma _H(u)=\sum _{\alpha }
h_{\alpha }\otimes \sigma (o_{\alpha })$. Or, if $X$ is a linear operator 
on $L(H)_{O(K)}$ then $\sigma _{L(H)}X$ is also a linear operator such that 
$$\sigma _{L(H)}X\left (\sum _{\alpha }h_{\alpha }\otimes o_{\alpha }\right )=
\sum _{\alpha }\sigma _H(X(h_{\alpha }))(1\otimes o_{\alpha })\,.$$
In addition, $\c X:=X\cdot \sigma _H$ is a unique 
sigma linear operator 
such that $\c X|_H=X|_H$, and we have the following identity 
$\sigma _H\cdot X=\sigma _{LH}(X)\cdot \sigma _H$. 
If there is no risk of confusion we will use just the notation $\sigma $. 
\medskip 

\begin{remark} 
Originally we developed in \cite{Ab2}  the 
contravariant version of the nilpotent Artin-Schreier theory, 
cf.\,the discussion 
in \cite{Ab11a}. The contravariant version 
uses similar relations $\sigma _{L(H)}(f)=
f  \circ e$ and the map $l_{   H}$ was defined via 
$\tau\mapsto (\id _{L(H)}\otimes\tau ) f  \circ (-f)$. In this case $l_{   H}$ 
determines the group homomorphism from $\Gamma _{K}$ to 
the opposite group $G^0(L(   H))$ (this group is isomorphic to 
$G(L(   H))$ via the map $g\mapsto g^{-1}$). 
The results from the papers \cite{Ab1, Ab2, Ab3} 
were obtained in terms of the contravariant version, but the  
results from \cite{Ab12, Ab13, Ab14, Ab15} used 
the covariant version.  
We can easily switch from one theory to another 
via the automorphism $-\op{id}_{L(   H)}$. 
\end{remark}

We consider $O_{sep}$ as a left $\Gamma _{   K}$-module 
via the action $o\mapsto \tau (o)$ with $o\in O_{sep}$ 
and $\tau\in\Gamma _{   K}$. This corresponds to our 
earlier agreement about the left action of 
the elements of $\Gamma _{   K}$ as endomoprhisms of $O_{sep}$.
As  a result we obtain the left $\Gamma _{   K}$-module 
structure on $H_{O_{sep}}$ by the use of the (left) action of 
$\Gamma _{   K}$ on $H$ via 
$h\mapsto l_{   H}(\tau )(h)$. 
 
Because 
$L(   H)_{O_{sep}}\subset \op{End}
_{O_{sep}}(   H_{O_{sep}})$ we can introduce for any 
$h\in H$, 
$\c F(h):=\exp (-f)(h)\in H_{O_{sep}}$. 

\begin{Prop}\label{P2.1} For any $h\in H$,  
 $\c F(h)\in (   H_{O_{sep}})^{\Gamma _{   K}}$.  
\end{Prop}
\begin{proof}
 Suppose $f=\sum _{\alpha }
 l_{\alpha }\otimes o_{\alpha }$, where 
 all $l_{\alpha }\in L(   H)$ and $o_{\alpha }\in O_{sep}$.
 
 If $\tau\in\Gamma _{   K}$ then 
 $$\tau (\c F(h))=(\tau\otimes\id _{O_{sep}})\left (\exp 
 \left (-\sum _{\alpha }l_{\alpha }\otimes 
 \tau (o_{\alpha })\right )(h)\right )=$$
 $$(\tau\otimes\id _{O_{sep}})\left (\exp \left (
 (\id _{L(H)}\otimes\tau )(-f)\right )(h)\right )=
 $$
 $$(\tau\otimes\id _{O_{sep}})\left (\exp 
 \left ((-l_{   H}(\tau ))  \circ (-f)\right )(h)\right )=
 $$
 $$(\tau\otimes\id _{O_{sep}})\left ((\exp 
 (-l_{H}(\tau ))  \cdot \exp (-f))h\right )=
 $$
 $$
 (\tau\otimes\id _{O_{sep}})(\pi _{H}(\tau ^{-1})\otimes 
 \id _{O_{sep}})\c F(h)=$$
 $$=\left (\tau\cdot \pi _H(\tau ^{-1})\otimes \id _{O_{sep}}
 \right )\c F(h)=\c F(h)\,.
 $$
\end{proof}

The elements $e$ and $f$ are not determined uniquely by $l_{   H}$.  
A pair $e'\in L(   H)_{O(   K)}$ and 
$f'\in L(   H)_{O_{sep}}$ give the same group 
epimorphism $l_{   H}$ iff 
there is $x\in L(   H)_{O(   K)}$ such that $e'=
\sigma (x)  \circ e  \circ (-x)$ 
and $f'=x  \circ f$. 

\medskip

\subsection{Special choice of $e\in L(H)_{O(   K)}$} \label{S2.2}

We can always assume (by replacing, if 
necessary, $   K$ by its finite 
unramified extension) that a  
uniformiser $t$ in $   K$ is such that  
$t^{e_0}=t_0$, where $e_0$ is 
the ramification index for $   K/   K_0$. Then 
$O(   K)=\underset{M}{\varprojlim}\, W_M(k)((\,\bar t))$, where 
$\bar t^{\,e_0}=\bar t_0$ and $\bar t$ is the 
Teichmuller representative of $t$.  We denote by $k$ the residue 
field of $   K$ and fix a choice of 
$\alpha _0=\alpha (k)\in W(k)$ such that its trace in the field 
extension $W(k)[1/p]/\Q _p$ equals 1. 

Let $\Z ^+(p):=\{ a\in\N \ |\ \op{gcd}(a,p)=1\}$ and 
$\Z ^0(p)=\Z ^+(p)\cup\{0\}$. 

\begin{definition} 
An element $e\in L(   H)_{O(   K)}$ is {\sc special} if 
$e=\sum _{a\in\Z ^0(p)}\bar t^{-a}l_{a0}$, where 
$l_{00}\in\alpha _0L(   H)$ and 
for all $a\in\Z ^+(p)$, $l_{a0}\in L(   H)_{W(k)}$.
\end{definition}

\begin{Lem} \label {L2.2} 
Suppose $e\in L(   H)_{O(   K)}$. Then there is 
$x\in L(   H)_{O(   K)}$ such that $\sigma (x)  \circ e  \circ (-x)$ 
is {\sc special}.
\end{Lem} 

\begin{proof} 
of $L$, 
 Use induction on $s$ to prove lemma modulo the ideals of $s$-th commutators 
 $C_s(L(   H)_{O(   K)})$. 
 
 If $s=1$ there is nothing to prove. 
 
 Suppose lemma is proved modulo $C_s(L(H)_{O(   K)})$.
 
 Then  there is 
 $x\in L(   H)_{O(   K)}$ such that 
 $$\sigma (x)  \circ e  \circ (-x)
 = \sum _{a\in\Z^0(p)}\bar t^{-a}l_{a0}+l\, ,$$ 
 where $l\in C_s(L(   H)_{O(   K)})$. 
 Using that 
 \begin{equation} \label{E2.2} O(   K)=(\sigma -
 \id _{O(   K)})O(   K)\oplus (\Z _p\alpha _0) 
 \oplus \left (\sum _{a\in\Z ^+(p)}W(k)\bar t^{-a}\right )
 \end{equation} 
 we obtain the existence of $x_s\in C_s(L(   H)_{O(   K)})$ 
 such that 
 $$l=\sigma (x_s)-x_s+\sum _{a\in\Z ^0(p)}\bar t^{-a}l_a\,,$$ 
 where $l_0\in\alpha _0L$ and 
 all remaining $l_a\in L_{W(k)}$. Then we can take $x'=x-x_s$ to 
 obtain our statement modulo $C_{s+1}(L(   H)_{O(   K)})$.
\end{proof}

\begin{Lem} \label{L2.3} Suppose $e\in L_{O(   K)}$ is {\sc special} 
and $x\in L_{O(   K)}$. Then the element $\sigma (x)  \circ e
  \circ (-x)$ is {\sc special} iff 
$x\in L$ (or, equivalently, if $\sigma x=x$). 
\end{Lem}

\begin{proof} 
 Use relation cf.\,\cite{BF},  
 $$\op{Ad}(\exp (X))\exp (Y)=\exp \left (\sum _{n\geqslant 0}
 \frac{1}{n!}\op{ad}^n(X)(Y)\right )\,\op{mod} (\op{deg}\,p),$$ 
 where $\op{Ad}(U)(V)=UVU^{-1}$ and $\op{ad}(U)(V)=[U,V]$. 
 Indeed, if $X=x\in L(H)$ and $Y=e$ then 
 $\sum _{n\geqslant 0}\op{ad}^n(x)(e)/n!$ is also special. 
 
 When proving the inverse statement we can use induction 
 modulo the ideals $C_s(L(   H))_{O(   K)}$, $s\geqslant 1$,  as follows. 
 
 Assume the lemma is proved modulo $C_s(L(   H)_{O(   K)})$. Then using  
 the {\sc if} part 
 we can assume that $x\in C_s(L(   H)_{O(   K)})$. Therefore, 
 $e+\sigma (x)-x$ is special modulo $C_{s+1}(L(   H))_{O(   K)}$, i.e. 
 $$\sigma (x)-x\in \alpha _0C_s(L)+\sum _{a\in\Z ^+(p)}
 t^{-a}C_s(L)_{W(k)}$$ 
 modulo $C_{s+1}(L(   H)_{O(   K)})$.  
 By \eqref{E2.2} this 
 implies the congruence 
 $$\sigma (x)\equiv x\,\op{mod}C_{s+1}(L(   H))_{O(   K)}\,,$$ 
 i.e. $x\in C_s(L(   H))\,\op{mod}\,C_{s+1}(L(   H)_{O(   K)})$.
 
 The lemma is proved. 
\end{proof}

\subsection{Construction of the $\phi $-module $M(   H)$} \label{S2.3}

Note that  $\pi _{H}=\exp (l_{   H})$, and therefore for all 
$\tau\in\Gamma _{   K}$, it holds 
$$\pi _{   H}(\tau )=
\exp (-f)  \cdot \exp (\id _{L(H)}\otimes \tau )f\,,$$ 
where   
$f\in L(   H)_{O_{sep}}\subset 
\op{End}_{O_{sep}}(   H_{O_{sep}})$, $\sigma _{L(H)}(f)=e\circ f$  and  
$\exp f\in\op{Aut}_{O_{sep}}(   H_{O_{sep}})$. 
\medskip

Let $\op{MF}^{et}_{   K}$ be the category of etale 
$\phi $-modules over $O(   K)$.  
Recall that 
its objects are $O(   K)$-modules of finite type $M$ 
together with a 
$\sigma $-linear morphism $\phi :M\To M$ such that 
its $O(   K)$-linear extension 
$\phi _{O(   K)}:M\otimes _{O(   K),\sigma }O(   K)\To  M$ 
is isomorphism. 
The correspondence $   H\mapsto M(   H):
=(   H\otimes _{\Z _p} O_{sep})^{\Gamma _{   K}}$, 
where $\phi :M(   H)\To M(   H)$ comes from the action of   
$\sigma $ on $O_{sep}$, determines the  
equivalence of the categories $\op{M\Gamma }_{   K}$ 
and $\op{MF}^{et}_{   K}$. 
\medskip 

Consider the $\Z _p$-linear embedding $\c F:   H\To H_{O_{sep}}$ 
from Sec.\,\ref{S2.1}. 
 Let 
$M(H)=\c F(H)_{O(K)}$. 
By extension of scalars 
we obtain natural 
isomorphisms (use that $O(K)$ and $O_{sep}$ are flat $\Z _p$-modules): 
$$ \c F\otimes\id _{O_{sep}}:   H_{O_{sep}}
\simeq M(   H)_{O_{sep}}$$  
$$\c F\otimes
\id _{O(K)}:H_{O(K)}\simeq M(H)\,,$$
which will be denoted for simplicity just by $\c F$. 

Note that by Prop.\,\ref{P2.1}, 
$M(   H)= (   H_{O_{sep}})^{\Gamma _{   K}}$. 
\medskip

The $O(K)$-module 
$M(H)$ is provided with the $\sigma $-linear morphism 
$\phi :M(   H)\To M(   H)$ 
uniquely determined for all $h\in   H$ via  
$$\phi (\c F(h))=\exp (-\sigma _{L(H)}f)(h)
=(\exp (-f)  \cdot \exp (-e))(h)=\c F(\exp (-e)(h))\,.$$ 

Consider the $O(   K)$-linear operator 
$$A:=\exp (-e)\in 
\exp (L(   H)_{O(   K)})\subset \op{Aut}_{O(   K)}   H_{O(   K)}\,.$$
Then $\c A:=A\cdot \sigma _H$ will be a unique 
$\sigma $-linear operator on $L(H)_{O(K)}$ such that 
$\c A|_{H}=A|_{H}$. 
Clearly, for any $u\in H_{O(K)}$, 
$$\phi (\c F(u))=\c F(\c A(u))\,,$$
and    
$M(   H)$ is etale 
$\phi $-module associated with the 
$\Z _p[\Gamma _{   K}]$-module $   H$. 

For example, suppose $pH=0$ and 
$\{h_i\ |\ 1\leqslant i\leqslant N\}$ is $\F _p$-basis of $H$. 
Then $\{\c F(h_i)\ |\ 1\leqslant i\leqslant N\}$ is a $   K$-basis 
for $M(H)$. If $A(h_i)=\exp (-e)(h_i)=\sum _ja_{ij}h_j$ with 
all $a_{ij}\in   K$ then 
$\phi (\c F(h_i))=\sum _ja_{ij}\c F(h_j)$, 
and $((a_{ij}))$ appears as the corresponding 
``Frobenius matrix''. 

It can be easily seen also that 
if $\{h_i\ |\ 1\leqslant i\leqslant N\}$ is a minimal system of 
$\Z _p$-generators in $H$ then 
$\{\c F(h_i)\ |\ 1\leqslant i\leqslant N\}$ is a minimal system of 
$O(K)$-generators in $M(H)$.  
\medskip

\subsection{The connection $\nabla $ on $M(   H)$} \label{S2.4}

The $(\phi , O(   K))$-module $M(   H)$ 
can be provided with a connection 
$\nabla :M(   H)\To M(   H)\otimes _{O(   K)}\Omega ^1_{O(   K)}$, 
\cite{Fo}.
This is an additive map uniquely determined by the properties:
\medskip 

a) for any $o\in O(   K)$ and $m\in    M(   H)$, 
$\nabla (mo)=\nabla (m)o+m\otimes d(o)$;
\medskip 

b) $\nabla  \cdot \phi =(\phi\otimes\phi )  \cdot\nabla $. 
\medskip 

By a), $\nabla $ is uniquely determined by its restriction 
to $\c F(H)$ (use that $M(H)=\c F(H)_{O(K)}$). 
Let $\wt{B}$ be 
a unique $O(   K)$-linear operator on $M(   H)$ such that  
for any $m\in\c F(H)$, 
$\nabla (m)=\wt{B}(m)d\bar t/\bar t$. Consider the $O(K)$-linear operator 
$B\in\op{End}H_{O(K)}$ such that for all $u\in H_{O(K)}$,  
$\wt{B}(\c F(u))=\c F(B(u))$. Obviously, 
$\wt B$ and $B$ can be recovered one from another.

Note that for any $\Z_p$-module $\c C$, the elements  $c\in \c C_{O(K)}$ 
appear uniquely in the form  
$c=\sum _n c_n\otimes \bar t^n$ with all $c_n\in \c C_{W(k)}$.  
Therefore, the map 
$\id _{\c C}\otimes \partial _{\bar t}: \c C _{O(K)}
\To \c C_{O(K)}$ 
such that 
$c\mapsto \sum _n c_n\otimes n\bar t^n$  
is well-defined. If $\c C\subset \c C_1$ is an embedding of $\Z _p$-modules then we have 
$(\id _{\c C_1}\otimes\partial _{\bar t})|_{\c C_{O(K)}}=\id _{\c C}\otimes\partial _{\bar t}$.

With the above notation:
\medskip 

1) for all $m\in M(H)$, $\nabla (m)=(\wt{B}
+\id _{\c F(H)}\otimes \partial _{\bar t})
(m)d\bar t/\bar t\,.$

2) for all $u\in H_{O(K)}$ and $X\in \op{End}H_{O(K)}$, 
$$(\id _H\otimes \partial _{\bar t})(X (u))=
(\id _{\op{End} H}\otimes \partial _{\bar t})(X)(u) +
X ((\id _{H}\otimes\partial _{\bar t})u)\, .$$

\begin{Prop} \label{P2.4} 
Let $C=-(\id _{\op{End}H}\otimes\partial _{\bar t})A  \cdot A^{-1}$ 
and let for 
any $n\geqslant 1$, 
$D^{(n)}=A  \cdot \sigma _{\op{End}H}(A)  \cdot\ldots  
\cdot \sigma _{\op{End}H}^{n-1}(A)$. 
Then  
$$B=\sum _{n\geqslant 0}p^n
\op{Ad}(D^{(n)})\sigma _{\op{End}H}^n(C)\,.$$ 
\end{Prop} 

\begin{remark} \label{r8} 
In Sect.\ref{S4} we will prove that  
$C,B\in L(H)_{O(K)}\subset\op{End}H_{O(K)}$. 
In particular, $\sigma _{\op{End}H}C=\sigma _{L(H)}C$ and the correspondence 
$u\mapsto (B+\op{id}_{L(H)}\otimes\partial _{\bar t})(u)d\bar t/\bar t$ 
gives a connection on $L(H)_{O(K)}$. 
\end{remark}

\begin{proof} Indeed, for any $u\in H_{O(K)}$, 
it holds 
$$(\nabla  \cdot\phi )(\c F(u))=(\nabla\cdot \c F\cdot \c A)u=((\wt B+\id _{\c F(H)}\otimes \partial _{\bar t})
\cdot \c F\cdot \c A)(u)d\bar t/\bar t$$

$$=(\c F\cdot (B+\id _H\otimes 
\partial _{\bar t})\cdot \c A)(u)d\bar t/\bar t\, .$$

On the other hand, 
$(\phi\otimes\phi )(\nabla (\c F(u)))=$
$$\phi (\c F(B+\id _H\otimes\partial _{\bar t})u)\phi (d\bar t/\bar t)=
p(\c F\cdot \c A\cdot (B+\id _H\otimes\partial _{\bar t}))(u)d\bar t/\bar t \,.$$
  
Equivalently, we have the following identity on $H_{O(K)}$,  
 $$(B\cdot A+(\id _H\otimes\partial _{\bar t})\cdot A)
 \cdot \sigma _H=
 pA  \cdot \sigma _H\cdot (B+\id _H
 \otimes \partial _{\bar t})
 $$
 Rewrite this equality as follows 
 $$(B\cdot A-pA\cdot \sigma _{\op{End}H}(B))\cdot \sigma _H=
 (-(\id _H\otimes\partial _{\bar t})\cdot A+
 A\cdot\sigma _H\cdot (\id _H\otimes p\partial _{\bar t})\cdot \sigma _H^{-1})\cdot\sigma _H\,.$$
 
 Notice that that on $H_{O(K)}$ we have 
 $\sigma _H\cdot (\id _H\otimes p\partial _{\bar t})\cdot \sigma _H^{-1}=
 \id _H\otimes \partial _{\bar t}$. As a result, the right hand side 
 equals $-(\id _{\op{End}H}\otimes\partial _{\bar t})(A)\cdot\sigma _H$.
 
 From $\sigma _H|_H=\id _H$ it follows that by restriction on $H$ we have 
 
 $$B\cdot A-pA\cdot \sigma _{\op{End}H}(B)=-(\id _{\op{End}H}\otimes\partial _{\bar t})A.$$
 
 By $O(K)$-linearity this identity holds on the whole $H_{O(K)}$.

 As a result, our identity appears in the form 
 $$(\id_H-p\op{Ad}(A)\cdot \sigma _{\op{End}H})B=
 -(\id _H\otimes\partial _{\bar t})(A)  \cdot A^{-1}\,.$$
It remains to recover $B$ using 
that
$$(\id _{H}-p\,\op{Ad}A  \cdot \sigma _{\op{End}H})^{-1}=
\sum _{n\geqslant 0}p^n\op{Ad} (D^{(n)})  \cdot \sigma _{\op{End}H}^n\,.$$ 
\end{proof}

\section{Ramification filtration 
modulo $p$-th commutators} \label{S3}

Recall that $K=k((t))\subset  K_0^{tr}$,  
$\pi _{H}(\Gamma _{K})=\exp (L(H))=I(H)\subset\op{Aut}_{\Z _p}(H)$.  
Note also that $O(K)=W(k)((\bar t))$, where 
${\bar t}^{\,e_0}=\bar t_0$ and 
$e_0$ is the ramification index of $K/K_0$.

\subsection{Lie algebra $\c L$ and identification $\eta _{<p}$} 
\label{S3.1}
 
Let  
$   K_{<p}$ be the maximal $p$-extension 
of $   K$ in $   K_0^{sep}$ 
with the Galois group of nilpotent class $<p$. 
Then $\c G _{<p}:=\op{Gal}(   K_{<p}/   K)=
\underset{M}\varprojlim 
\,\Gamma _{   K}/\Gamma _{   K}^{p^M}C_p(\Gamma _{   K})$.

Let  
$\wt{\c L}_{W(k)}$ be a profinite free Lie $W(k)$-algebra with the set 
of topological generators 
$\{D_{0}\}\cup\{D_{an}\ |\ a\in\Z ^+(p), n\in\Z/N_0\}$. 
Let $\c L_{W(k)}=\wt{  \c L}_{W(k)}/C_p(\wt{  \c L}_{W(k)})$, 
where $C_p(\wt{  \c L}_{W(k)})$ is 
the ideal of $p$-th commutators. 
Define the $\sigma $-linear action on $  \c L_{W(k)}$ 
via $D_{an}\mapsto D_{a,n+1}$ and $D_0\mapsto D_0$, denote this 
action by the same symbol $\sigma $,  
and set $  \c L= \c L_{W(k)}|_{\sigma =\id }$. 

Fix $\alpha _0=\alpha _0(k)\in W(k)$ such that 
the trace of $\alpha _0$ in the field extension 
$W(k)[1/p]/\Q _p$ equals 1. 
For any $n\in\Z /N_0$, set 
$D_{0n}=(\sigma ^n\alpha _0)D_0$.

We are going to apply the profinite version of 
the covariant nilpotent Artin-Schreier theory to the 
Lie algebra $\c L$ and the special element    
$e_{<p}=\sum \limits _{a\in\Z ^0(p)} 
\bar t^{-a}D_{a0}\in  \c L\hat\otimes 
O(K)$.
In other words, if we fix   
$$   f_{<p}\in \{f\in \c L\hat
\otimes O_{sep}\ |\ \sigma _{\c L}(f)=
e  \circ f\}\ne\emptyset \, .$$
then the map 
$\eta _{<p}:=\pi _{f_{<p}}(e_{<p})$ 
given by the correspondence 
$\tau\mapsto (-f_{<p})  \circ (\id _{\c L}\otimes\tau )f_{<p}$ 
induces the group isomorphism 
$\bar\eta _{<p}:\Gamma _{<p}\simeq G(\c L)$. 

The following property is an easy consequence of the above construction.  

\begin{Prop} \label{P3.1} Suppose 
$e\in L(   H)_{O(   K)}$ is special and given with notation from  
the definition from Sect.\,\ref{S2.2}. 
Then the map $\log\pi _{   H}:\c G_{<p}\To G(L(   H))$ is given 
via the correspondences $D_{a0}\mapsto l_{a0}$ 
(and $D_{an}\mapsto \sigma _{L(H)}^n(l_{a0})$) for all $a\in\Z ^0(p)$.
\end{Prop}

\subsection{The ramification ideals $\c L^{(v)}$} \label{S3.2}

For $v\geqslant 0$, denote by $\c G ^{(v)}_{<p}$ the image 
of $\Gamma _{   K}^{(v)}$ in $\c G _{<p}$. Then 
$\bar{\eta }_{<p}(\c G ^{(v)}_{<p})=G(\c L^{(v)})$, where 
$\c L^{(v)}$ is an ideal in $\c L$. 
The images $\c L^{(v)}(M)$ of the ideals $\c L^{(v)}$ 
in the quotients 
$\c L/p^M\c L$ for all $M\in\N $
were explicitly described in \cite{Ab3}. By going to the 
projective limit on $M$ this description can be presented as follows.  
\medskip

\begin{definition} 
 Let $\bar n=(n_1,\dots ,n_s)$ with $s\geqslant 1$. 
 Suppose there is a partition  
 $0=i_0<i_1<\dots <i_r=s$ such that 
 for $i_j<u\leqslant i_{j+1}$, it holds  
 $n_u=m_{j+1}$ and $m_1>m_2>\dots >m_r$. Then  set  
 $$\eta (\bar n)=\frac{1}{(i_1-i_0)!\dots (i_r-i_{r-1})!} $$ 
 If such a partition does not exist we set $\eta (\bar n)=0$. 
\end{definition}

For $s\in\N $, $\bar a=(a_1,\dots ,a_s)\in\Z ^0(p)^s$ and 
$\bar n=(n_1,\dots ,n_s)\in\Z ^s$,   
set 
$$[D_{\bar a,\bar n}]=[\ldots [D_{a_1n_1},D_{a_2n_2}],
\dots ,D_{a_sn_s}]\,.$$ 

For $\alpha\geqslant 0$ and $N\in\Z _{\geqslant 0} $, introduce 
$\c F^0_{\alpha ,-N}\in{   L}_{W(k)}$ such that 
$$\c F^0_{\alpha ,-N}=\sum _{\substack {1\leqslant s<p\\
\gamma (\bar a, \bar n)=\alpha }}
a_1\eta (\bar n)p^{n_1}[D_{\bar a,\bar n}]\, .$$

Here $n_1\geqslant 0$, all $n_i\geqslant -N$ and 
$\gamma (\bar a,\bar n)=
a_1p^{n_1}+a_2p^{n_2}+\dots +a_sp^{n_s}$\,.
\medskip 

Note that the non-zero terms in the above expression for 
$\c F^0_{\alpha , -N}$ can appear only if 
$n_1\geqslant n_2\geqslant\ldots \geqslant n_s$ and 
$\alpha $ has at least 
one presentation in the form $\gamma (\bar a,\bar n)$. 
\medskip 

Denote by $\c I^{(v)}[N]$ the minimal closed ideal in 
$\c L$ such that its extension of scalars $\c I^{(v)}[N]_{W(k)}$ 
contain all $\c F^0_{\alpha ,-N}$ with $\alpha\geqslant v$. 

Our result from \cite{Ab3} about explicit generators of 
the ideal ${\c L}^{(v)}$ can be stated 
in the following form. 

\begin{Thm} \label{T3.2}
For any $v>0$ and $M\in\N $, there is  
$\wt{N}(v,M)\in\N $ such that if   
$N\geqslant \wt{N}(v,M)$, then the images of the ideals $\c L^{(v)}$ and $\c I^{(v)}[N]$ 
in $\c L/p^M$ coincide. 
\end{Thm}

\subsection{Some relations} \label{S3.3} 

Let $A(  \c L)$ be 
the enveloping algebra of $ \c L$ and $\wt{A}(\c L)=
A(\c L)/J(\c L)^p$, where 
$J(\c L)$ is the augmentation ideal in $A(\c L)$. Note that there is a natural 
embedding of $\Z _p$-modules $\c L\subset \wt{A}(\c L)$. 

Let $A_{<p}=\exp (-e_{<p})\in \wt{A}(\c L)_{O(K)}$ and  
$C_{<p}=-
(\id _{\wt{A}(\c L)}\otimes \partial _{\bar t})A_{<p}  \cdot A_{<p}^{-1}$. 

For $s\geqslant 1$, set 
$\bar 0_s=(\underset{\text{$s$ times}}
{\underbrace{0,\ldots ,0}})$.

\begin{Prop} \label{P3.3} 
 Let   
 $D^{(m)}_{<p}:=A_{<p}
   \cdot \sigma _{\wt{A}(\c L)}(A_{<p})\cdot \ldots   \cdot  
 \sigma _{\wt{A}(\c L)}^{m-1}(A_{<p})$, where $m\geqslant 1$.  
 Then we have the following 
 relations:
\begin{equation} \label{E3.1} 
C_{<p}=
 \sum _{s\geqslant 1,\bar a }
 a_1\eta (\bar 0_s)
 [D_{\bar a, \bar 0_s}]
 \bar t^{-\gamma (\bar a, \bar 0_s)}\, ;
 \end{equation}
 \begin{equation} \label{E3.2} 
B_{<p}:=\sum _{n\geqslant 0}p^n
\op{Ad}(D ^{(n)}_{<p})(\sigma _{\c L}^n(C_{<p}))=
\sum _{\alpha >0}\c F^0_{\alpha ,0}\bar t^{-\alpha }\, ;
\end{equation} 

\begin{equation} \label{E3.3} 
\op{Ad}\,\sigma _{\wt{A}(\c L)}^{-m}(D ^{(m)}_{<p})
(B _{<p})=
\sum _{\alpha >0}\c F^0_{\alpha ,-m}\bar t^{-\alpha }\, .
\end{equation}

\end{Prop}

\begin{proof} For \eqref{E3.1} use, cf.\,\cite{BF}, 
theorem 4.22, to obtain 
$$d\exp (-e_{<p})  \cdot \exp (e_{<p})=
\sum _{k\geqslant 1}\frac{1}{k!}
(-\op{ad} e_{<p})^{k-1}(-d\, e_{<p})\,$$
and note that 
$(-\op{ad}e_{<p})^{k-1}(de_{<p})=$
$$(-1)^{k-1}[\underset{\text{$k-1$ times }}
{\underbrace{e_{<p},\cdots , [e_{<p}}},de_{<p}]\ldots ]= 
[\dots [de_{<p},
\underset{\text{$k-1$ times }}
{\underbrace{e_{<p}],\dots , e_{<p}}}]$$

For \eqref{E3.2} we need the following relation 
cf.\,\cite{BF}, Sect.\,4.4
$$\exp (X)  \cdot Y  \cdot \exp (-X)= 
\sum _{n\geqslant 0}\frac{1}{n!}\op{ad}^n(X)(Y)\, .$$

After applying this relation to the summand with $m=1$ we obtain 
$$p\op{Ad}\,D_{<p}^{(1)}(\sigma _{\c L}C_{<p})=
\exp (-e_{<p})  \cdot \sigma _{\c L}(C_{<p})  \cdot 
\exp (e_{<p})=$$
$$=\sum _{s\geqslant 0}\eta (\bar 0_s)(-1)^s\op{ad}^s(e_{<p})
(C_{<p})=$$
$$\sum _{\bar n\geqslant 0, \bar a}\,\,\sum _{n_1=0}^{n_1=1}
\eta (\bar n)p^{n_1}[D_{\bar a,\bar n}]
\bar t^{-\gamma (\bar a, \bar n)}\, .$$
Repeating this procedure we obtain relation \eqref{E3.2}.

Similar calculations prove the remaining item \eqref{E3.3}.  
\end{proof}

\section {Proof of Theorem \ref{T1.1}} \label{S4}

Recall briefly what we've already achieved.  

The field $K=k((t))$ is tamely ramified extension of $K_0=k_0((t_0))$, 
where $t^{e_0}=t_0$, $H:=H_0|_{\Gamma _{K}}$ and the corresponding 
group epimorphism $\pi _H:\Gamma _K\To I(H)\subset  \op{Aut}_{\Z _p}H$ 
is such that $I(H)=\exp (L(H))$, where $L(H)\subset\op{End}_{\Z _p}H$ and $L(H)^p=0$. 

Applying formalism of nilpotent Artin-Schreier theory we obtained 
a special $e=\sum _{a\in\Z ^0(p)}\bar t^{-a}l_{a0}\in L(H)_{O(K)}$ and $f\in L(H)_{O_{sep}}$ 
such that 
$$\exp (\sigma _{L(H)}f)=\exp (e)\cdot \exp (f)$$ 
and 
for any $\tau \in \Gamma _{K}$, $\pi _H(\tau )
=\exp (-f)\cdot \exp(\id _{L(H)}\otimes \tau )f$. 

We used $O(K)$-linear operator $\c F=\exp (-f)$ to 
introduce $O(K)$-module $M(H):=\c F(H_{O(K)})$. Let $A=\exp (-e)$ and 
let ${\c A}$ be a unique 
$\sigma $-linear operator on $H_{O(K)}$ such that for any $h\in H$, 
$\c A (h)=A(h)$. Then the $\sigma $-linear $\phi :M(H)\To M(H)$ is 
such that for any $u\in L(H)_{O(K)}$, 
$\phi (\c F(u))=\c F(\c A(u))$. As a result, 
we obtain the structure of etale $(\phi , O(K))$-module on $M(H)$ related to the $\Gamma _K$-module $H$.

Let $\nabla $ be the connection on $M(H)$ from Sect.\,\ref{S2.4}, 
and let $\wt{B}$ be the 
$O(K)$-linear operator on $M(H)$ uniquely determined by the condition: 
for any $m\in\c F(H)$, 
$\nabla (m)=\wt{B}(m)d\bar t/\bar t$. 
Then for any 
$u\in M(H)$, 
$$\nabla (u)=(\wt{B}+\id _{\c F(H)}\otimes \partial _{\bar t})(u)
d\bar t/\bar t\,$$
and we introduce the differential forms 
$$\wt{\Omega }[N]=
\phi ^N\wt{B}\phi ^{-N}d\bar t/\bar t\in 
\op{End}\,M(H)_{O(K^{rad})}\otimes\Omega ^1_{O(K)}\,.$$

Finally, define the $O(K)$-linear operator $B$ on $H$ by setting 
for any $u\in H_{O(K)}$, $\c F(B(u))=\wt{B}(\c F(u))$, and 
transfer $\wt{\Omega}[N]$ to $\op{End}H_{O(K^{rad})}$ in the following form  
$$\Omega [N]=
\op{Ad}({\c A}^N)(B)d\bar t/\bar t\in \op{End}(H)_{O(K^{rad})}\otimes 
_{O(K)}\Omega ^1_{O(K)}\, .$$ 
\begin{remark}
 Obviously we have the following identification  
 $$\op{End}(H)_{O(K^{rad})}\otimes _{O(K)}\Omega ^1_{O(K)}=
 \op{End}(H_0)_{O(K_0^{rad})}\otimes _{O(K_0)}\Omega ^1_{O(K_0)}\,.$$
\end{remark}

Recall that $\c A=A\cdot \sigma _H$, where $A=\exp (-e)$. 

\begin{Lem}\label{L4.1} If 
    $\c D^{(N)}=\sigma _{\op{End}H}^{-N}(A)\cdot   
    \ldots \cdot \sigma _{\op{End}H}^{-1}(A)$ then 
    $$(\sigma _{\op{End}H}^{-N}\cdot \op{Ad}(\c A^N))(B)=
    \op{Ad}\c D^{(N)}(B),\,.$$
\end{Lem}

\begin{proof} Use induction on $N\geqslant 0$. 
If $N=0$ there is nothing to prove.  
 Suppose lemma is proved for $N\geqslant 0$. Then 
 $$(\sigma _{\op{End}H}^{-(N+1)}\cdot\op{Ad}(\c A^{N+1}))(B)=
 \sigma _{\op{End}H}^{-(N+1)}(\c A\cdot \op{Ad}(\c A^N)
 (B)\cdot\c A^{-1})=$$
 $$\sigma _{\op{End}H}^{-(N+1)}(\c A\cdot 
 (\sigma _{\op{End}H}^N\cdot\op{Ad}(\c D^{(N)})(B))\cdot \c A^{-1})=$$
 $$\sigma _{\op{End}H}^{-(N+1)}(A\cdot \sigma _H\cdot 
 (\sigma _{\op{End}H}^{N}\op{Ad}(\c D^{(N)})(B))\cdot \sigma _H^{-1}\cdot A^{-1})
 =$$
 $$\sigma _{\op{End}H}^{-(N+1)}(A\cdot  
 (\sigma _{\op{End}H}^{N+1}\op{Ad}(\c D^{(N)})(B))\cdot A^{-1})
 =$$
 $$\sigma _{\op{End}H}^{-(N+1)}(A)\cdot \op{Ad}(\c D^{(N)})(B)\cdot 
 \sigma _{\op{End}H}^{-(N+1)}(A^{-1})=
 \op{Ad}(\c D^{(N+1)})(B)$$
 The lemma is proved. 
 \end{proof} 
 
 Under the projection $\log\bar{\pi} _H:\c G_{<p}\To G(L(H))$ we have:
 \medskip 
  
 $D_{an}\mapsto l_{an}=\sigma _{L(H)}^nl_{a0}$, 
 $e_{<p}\mapsto e$, $f_{<p}\mapsto f$, 
 $A_{<p}\mapsto A$,  $C_{<p}\mapsto C$, 
 \medskip 
 
 $B_{<p}\mapsto B$ 
 and $\sigma ^{-m}_{\wt{\c A}(\c L)}D^{(m)}_{<p}\mapsto \c D^{(m)}$. 
 
 \begin{remark} 
 Because $\log \bar\pi _H(\c L_{<p})=L(H)$ we obtain the statement from Remark \ref{r8}.
 \end{remark} 
 
 As a result, our differential form appears as 
 the image of 
 $\sum \c F^0_{\alpha ,-N}\bar t^{-\alpha }
 d\bar t/\bar t$.
 \medskip 
 
 It remains to notice that when getting back to the field $K_0$, 
 we have $d\bar t/\bar t=e_0^{-1}d\bar t_0/\bar t_0$, 
 $\bar t^{-\alpha }=\bar t^{-\alpha /e_0}$, 
 $\Gamma _K^{(\alpha )}=\Gamma _{K_0}^{(\alpha /e_0)}$ and  
 $\pi _H|_{\c I}=\pi _{H_0}|_{\c I}$. 
 \medskip 
 
 Theorem \ref{T1.1} is proved. 

\begin{remark}
 a) The conjugacy class of the 
 differential form $\Omega [N]$ does not 
 depend on a choice of a special form for $e$.  
\medskip 
 
 b) It would be very interesting to verify whether our 
 results could be established in the case of  
 $\Gamma _{K}$-modules which 
 do not satisfy the $\op{Lie}$ condition, e.g. for the 
 $\Gamma _{K}$-module from \cite{Im} (the case $n=p$ 
 in the notation of that paper). 
\end{remark}

\section{Mixed characteristic} \label{S5}

Let 
$ E_0$ be a complete discrete valuation field of 
characteristic 0 with finite residue field 
$k_0$ of characteristic $p$. 
Let $\bar E_0$ be an algebraic closure of $E_0$ and for any 
field $E$ such that $E_0\subset E\subset \bar E_0$, set 
$\op{Gal}(\bar E_0/E)=\Gamma _E$. 
Suppose that $E_0$ contains a primitive $p$-th root of 
unity $\zeta _1$.

We are going to develop an analog of the above characterisitc $p$ theory in the context of finite 
$\F _p[\Gamma _{E_0}]$-modules $H_{E_0}$ satisfying 
an analogue of 
the Lie condition from Sect.\,\ref{S2.1}:
\medskip 

{\it if $\pi _{H_{E_0}}:\Gamma _{E_0}\To\op{Aut}_{\F _p}(H_{E_0})$ 
determines a $\Gamma _{E_0}$-action on $H_{E_0}$ then there is a  
Lie $\F _p$-subalgebra $L(H_{E_0})\subset \op{End}_{\F _p}(H_{E_0})$ such that 
$L(H_{E_0})^p=0$ and $\exp (L(H_{E_0}))=\pi _{H_{E_0}}(I)$, where 
$I$ is the wild ramification subgroup in $\Gamma _{E_0}$.}
\medskip 

\begin{remark}
Contrary to the characteristic $p$ case we restrict ourselves to  
the Galois modules killed by $p$ because the theory 
from   \cite{Ab12, Ab13} is developed recently only under that assumption.  
\end{remark}

Fix a choice of a uniformising element $\pi _0$ in $E_0$. 

Let 
$\wt{E}_0=E_0(\{\pi^{(n)} _0)\ |\ n\in\Z _{\geqslant 0}\})
\subset \bar E_0$, where 
$\pi ^{(0)}_0=\pi _0$ and for all $n\in\N $, 
$\pi ^{(n)p}_0=\pi ^{(n-1)}_0$. 
The field-of-norms functor $X$ 
provides us with:
\medskip 

-- a complete discrete  
valuation field  $X(\wt{E}_0)=K_0$  
of characteristic $p$ with residue field $k_0$ and a  
fixed uniformizer 
$t_0=\varprojlim \pi ^{(0)}_n$;
\medskip

-- 
an identification of $\Gamma _{K_0}=
\Gal (K_0^{sep}/K_0)$ with ${\Gamma }_{\wt{E}_0}
\subset\Gamma _{E_0}$. 
\medskip 

Let $E$ be a finite tamely ramified extension of $E_0$ 
in $\bar E_0$ 
such that $\pi _{H_{E_0}} (\Gamma _E)=I(H_0)$. 
By replacing $E$ with a suitable finite unramified 
extension we can assume that $E$ has uniformiser $\pi $ such that 
$\pi ^{e_0}=\pi _0$. Let $k$ be the residue field of $E$. 

It is easy to see that the field 
$\wt{E}:=E\wt{E}_0$ appears in the form 
$E(\{\pi^{(n)}\ |\ n\geqslant 0\})$, where 
$\pi ^{(0)}=\pi $,  
$\pi ^{(n)p}=\pi ^{(n-1)}$ and for all $n$, 
$\pi ^{(n)e_0}=\pi ^{(n)}_0$. 
In particular, $K:=X(\wt{E})=k((t))$, where 
$t=\varprojlim \pi ^{(n)}$ is 
uniformiser such that $t^{e_0}=t_0$.

Let $\c G_{<p}=\Gamma _K/\Gamma _K^pC_p(\Gamma _K)$ 
and $\Gamma _{<p}=\Gamma _E/\Gamma _E^pC_p(\Gamma _E)$. 
According to \cite{Ab12, Ab13} we have the following natural exact sequence 
$$\c G_{<p}\To \Gamma _{<p}\To \langle \tau _0
\rangle ^{\Z /p}\To 1\,,$$
where $\tau _0\in\op{Gal}(E(\pi ^{(1)})/E)$ is such that 
$\tau _0(\pi ^{(1)})=\zeta _1\pi ^{(1)}$. We can use the identification 
$\bar\eta _{<p}:\c G _{<p}\simeq G(\c L)$ from Sec.\,\ref{S3.1} 
obtained via the special element $e_{<p}$ and the corresponding 
$f_{<p}$ such that $\sigma _{L(H)}(f_{<p})=e_{<p}\circ f_{<p}$. 

Then we can use the equivalence of categories from Sec.\ref{S1.3} to identify 
$\Gamma _{<p}$ with $G(L)$ where $L$ is 
a profinite Lie $\F _p$-algebra included into the following exact sequence 
\begin{equation} \label{E5.1}
 \c L\To L\To \F _p\tau _0\To 0\, .
\end{equation}

When studying the structure of \eqref{E5.1} in \cite{Ab13} 
we proved that $\tau _0$ can be replaced by a suitable 
$h _0\in\op{Aut}K$. 
More precisely, 
suppose  
$$\zeta _1\equiv 1+\sum _{j\geqslant 0}[\beta _j]\pi _0^{(e_0^*/p)+j}\,\op{mod}\,p$$
with Teichm\" uller representatives $[\beta _j]$ of $\beta _j\in k$ 
and $e^*=ep/(p-1)$, where $e$ is the ramification index for $E/\Q _p$. Then 
$h_{0}$ can be defined as follows: $h_{0}|_k=\id _k$ and 
$$h_{0}(t)=t\left (1+\sum _{j\geqslant 0}\beta _j^pt^{e^*+pj}\right )=
t\wt{\exp}(\omega (t)^p)\,,$$
where $\wt{\exp}$ is the truncated exponential and 
$\omega (t)\in t^{e^*/p}k[[t]]^*$. 
\medskip

This allowed us to apply 
formalism of the nilpotent Artin-Schreier theory 
to specify ``good'' lifts $\tau _{<p}$ of $\tau _0$ to $L$ 
(what is equivalent to specifying "good lifts" $h_{<p}$  of $h_0$).

In particular, we obtained in \cite{Ab12, Ab13} the following description of the image 
$\bar L\subset L$ of $\c L$ from exact sequence \eqref{E5.1}. 
Introduce the weight function $\op{wt}$ on $\c L_{k}$ 
by setting $\op{wt}(D_{an})=s\in\N $ iff $(s-1)e^*\leqslant a<se^*$. Then 
$$\op{Ker}(\c L\To L)=\c L(p)=\{l\in \c L\ |\ \op{wt}(l)\geqslant p\}\,,$$
$\bar L=\c L/\c L(p)$ and $\c L\To \bar L$ is 
the natural projection. 

If $h_{<p}$ is a lift of $h_0$ to $K_{<p}$ then it is uniquely determined by $c=c(h_{<p})\in L_{K}$ 
such that 
$$(\id _{\c L_{<p}}\otimes h_{<p})f=c\circ (\op{Ad}(h_{<p})\otimes\id _{K_{<p}})f\, .$$
This allowed us to describe the corresponding action of the group 
$\langle h_{<p}\rangle ^{\Z /p}$ on $f$ as an action of an infinitesimal group scheme of order $p$. The differential of 
this action is given by the "linear part" $c_1\in \bar L_{K}$ of $c$ which could be described by a suitable 
recurrent procedure. Finally, we proved that $c_1(0)\in \bar L_{k}$ 
(where $c_1=\sum _{n\in\Z}c_1(n)t^n$ with all $c_1(n)\in\bar L_{k}$) is an absolute invariant of the lift $h_{<p}$.

Consider the expansion $\omega (t)^p=\sum _{j\geqslant 0}A_jt^{e^*+pj}$, 
$A_j\in k$.

Denote by $\bar L^{(e^*)}\subset \bar L$ 
the image of the ramification subgroup $\c L^{(e^*)}$ in $L$, cf. \eqref{E5.1}.  
Then Theorem 5.1 from \cite{Ab13} states: 

 $\bullet $\ the lift $\tau _{<p}$ of $\tau _0$ is ``good''   
 iff  the value $c_1(0)\in \c L_{k}$ of the differential 
 $d\tau _{<p}$ at 0
 satisfies the following congruence 
 $$c_1(0)\equiv \sum _{j\geqslant 0}
\sum _{i\geqslant 0}\sigma ^i(A_j   \c F^0_{e^*+pj,-i})                          
\,\op{mod}\,\bar{L}_k^{(e^*)}\, .$$ 
It remains to note that for $i,j\gg 0$, 
$\c F^0_{e^*+pj,-i}\in \bar L_k^{(e^*)}$ and 
the right-hand double sum contains only finitely many non-zero terms modulo $\bar L_{k}^{(e^*)}$.  
It can be rewritten also in the following form 
$$\sum _{i,j\geqslant 0}\op{Res}\left (
\sigma ^i(A_jt^{e^*+pj}\cdot 
\sigma ^{-i}\Omega _{<p}[i])\right )\, ,$$
and the image of this expression in $L(H)_k$ equals  
$$\sum _{i\geqslant 0} \op{Res}\left (\sigma ^{i+1}\omega (t)\cdot \Omega [i]\right )\,.$$

{\sc Acknowledgements.} The author expresses his deep 
gratitude to Prof.\,E.\,Khukhro for helpful discussions.

\end{document}